\documentclass[a4paper,oneside,10pt]{amsart}

\usepackage{a4wide}
\usepackage[T1]{fontenc}
\usepackage[ansinew]{inputenc}
\usepackage{lmodern} 
\usepackage{graphicx}
\usepackage{amsmath}
\allowdisplaybreaks
\usepackage{amsthm}
\usepackage{amsfonts}
\usepackage{amssymb}
\usepackage{setspace}
\usepackage{mathrsfs}
\usepackage[all]{xy}
\usepackage{enumerate}
\usepackage{xcolor}
\usepackage{catchfile}
\usepackage{autobreak}
\usepackage[textsize=tiny,shadow]{todonotes}
\usepackage{changes}

\theoremstyle{plain}
\newtheorem{theorem}{Theorem}[section]
\newtheorem{corollary}[theorem]{Corollary}
\newtheorem{proposition}[theorem]{Proposition}
\newtheorem{lemma}[theorem]{Lemma}
\theoremstyle{definition}
\newtheorem{remark}[theorem]{Remark}

\newtheorem{example}[theorem]{Example}
\newtheorem{definition}[theorem]{Definition}

\newtheorem*{proposition*}{Proposition}
\newtheorem*{definition*}{Definition}
\numberwithin{equation}{section}
\theoremstyle{plain}
\newtheorem*{theorem*}{Theorem}

\newenvironment{abc}{\begin{enumerate}[{\rm (a)}]}{\end{enumerate}}

\newenvironment{iiv}{\begin{enumerate}[{\rm (i)}]}{\end{enumerate}}

\newtheorem{assumptionA}{Assumption}{\bfseries}{\upshape}

\def\dom{\mathrm{D}}
\def\dd{\mathrm{d}}

\def\Id{\mathrm{I}}
\def\ee{\mathrm{e}}
\def\RR{\mathbb{R}}
\def\NN{\mathbb{N}} 
\def\CC{\mathbb{C}}
\def\LLL{\mathscr{L}}

\def\Lip{\mathrm{Lip}}

\def\CD{\mathrm{CD}}

\newcommand{\abs}[1]{\left| #1 \right|} 
\newcommand{\norm}[1]{\|#1\|}
\newcommand{\from}{\colon}
\DeclareMathOperator{\TextRe}{Re}

\renewcommand{\Re}{\TextRe}



\begin{document}
\title{Non-autonomous Desch-Schappacher perturbations} 

\author{Christian Budde}
\address{University of the Free State, Department of Mathematics and Applied Mathematics, P.O. Box 339, 9300 Bloemfontein, South Africa}
\email{BuddeCJ@ufs.ac.za}

\author{Christian Seifert}
\address{Christian-Albrechts-Universit\"at zu Kiel, Mathematisches Seminar, Heinrich-Hecht-Platz 6, 24118 Kiel, Germany and
Technische Universit\"at Hamburg, Institut f\"ur Mathematik, Am Schwarzenberg-Campus 3, 21073 Hamburg, Germany}
\email{christian.seifert@tuhh.de}

\begin{abstract}                                                                         
We consider time-dependent Desch--Schappacher perturbations of non-autonomous abstract Cauchy problems and apply our result to non-autonomous uniformly strongly elliptic differential operators on $\mathrm{L}^p$-spaces.
\end{abstract}


\date{}
\maketitle

\section*{Introduction}

For many processes in sciences, the coefficients of the partial differential equation describing a dynamical system as well as the boundary conditions of it may vary with time. In such cases one speaks of \emph{non-autonomous} (or time-varying) evolution equations. For applications of non-autonomous dynamical systems we refer for examples to \cite{KP2013}. From an operator theoretical point of view one considers families of Banach space operators which depend on the time parameter and studies the associated non-autonomous abstract Cauchy problem. In particular, for fixed $T>0$ and a family of linear (and typically unbounded) operators $(A(t),\dom(A(t)))_{t\in[0,T]}$ on a Banach space $X$ one considers the non-autonomous abstract Cauchy problem given by
\begin{align}\label{eqn:nACP}\tag{nACP}
\begin{split}
\dot{u}(t) &= A(t)u(t),\quad T\geq t\geq s\geq0,\\
u(s) &= x\in X.
\end{split}
\end{align}
While in the autonomous case operator semigroups yield fundamental solutions and thus provide an appropriate solution concept, for the non-autonomous case one needs to make use of so-called evolution families $(U(t,s))_{t\geq s}$, which give rise to a notion of well-posedness \cite[Sect.~3.2]{NickelPhD}. Unfortunately, the existence of solutions of \eqref{eqn:nACP} is not as simple as in the autonomous case which, among others, comes from the lack of a differentiable structure of an evolution family as opposed to $C_0$-semigroups \cite[Ex.~3.5]{NickelPhD}. Nevertheless, there are several independent and technical attempts to find sufficient conditions for (unique) solutions of \eqref{eqn:nACP}, for example by Acquistapace and Terreni \cite{Acquistapace1987} or Kato and Tanabe \cite{K1953,T1979}. 

\medskip
Perturbation theory is a powerful tool in order to study existence, uniqueness and other qualitative properties and allows for a more general and abstract view on non-autonomous abstract Cauchy problems. More precisely, additionally given a family $(B(t),\dom(B(t)))_{t\in[0,T]}$ of linear operators in $X$ one studies the perturbed Cauchy problem
\begin{align*}
\dot{u}(t) & = (A(t)+B(t))u(t),\quad T\geq t\geq s\geq0,\\
u(s) & = x\in X,
\end{align*}
trying to make use of information of the unperturbed Cauchy problem \eqref{eqn:nACP}. Time-dependent perturbations of evolution equations have attracted a lot of interest in the past, see for example \cite{Evans1976,Rhandi1997}. 

\medskip
In this article, we consider non-autonomous Desch--Schappacher perturbations, i.e., we demand that the perturbation is merely bounded with respect to suitable extrapolation spaces. In our context, there are several technicalities to deal with. First of all, one has to consider extrapolation of evolution families. For this purpose, we will assume that the extrapolation spaces regarding the given family $(A(t))_{t\in[0,T]}$ of operators are uniformly equivalent. This assumption is reasonable and has been made by several authors, see for example \cite{Ama1988,JL2021}. We notice that the uniformity of the extrapolation spaces does not imply that the domains $\dom(A(t))$ are independent of time as well. The usage of extrapolation spaces first yields evolution families in this extrapolation space. In order to obtain an evolution family for the perturbed system in the original space $X$ we ask for a bit more regularity of the perturbation, namely by suitable mapping properties into the so-called Favard class. Our perturbation result is motivated by the work of Rhandi \cite{Rhandi1997} for autonomous $A$ and non-autonomous  perturbation. In order to treat the non-autonomous case, we make use of the work of Bertoni \cite{B2002,B2005,B2014}, who studied perturbations of so-called $\CD$-systems.

\medskip
It is worth to review the already existing literature regarding non-autonomous Desch--Schappacher perturbations. In \cite{HP1994}, Hinrichsen and Pritchard consider stability for time-infinite systems of closed-loop output-feedback type by means of so-called weak evolution operators and piecewise continuous operator functions. They are particularly interested in perturbed system equations of the form
\[
\dot{x}(t)=\left[A(t)+D(t)\Delta(t)E(t)\right]x(t),\quad t\geq0.
\]
They obtain non-autonomous Desch--Schappacher type perturbations by specifying $\Delta(t)=\Id$ and $E(t)=\Id$. A few years later, Schnaubelt also studied time-infinite systems of closed-loop output-feedback type in \cite{S2002} where he makes use of a Duhamel type formula in an approximative sense due to the lack of extrapolation theory for non-autonomous evolution equations. Within this framework he gets joint nonautonomous extension of the Desch--Schappacher and Miyadera perturbation theorem, cf. \cite[Rem.~4.6(a)]{S2002}. Last but not least, we want to mention the joint work of Maniar and Schnaubelt \cite{MS2008} where they follow the most general approach due to Acquistapace and Terreni \cite{Acquistapace1987} to develop a non-autonomous Desch--Schappacher perturbation theorem. In particular, the domains of the operators $(A(t),\dom(A(t)))$ are allowed to vary in time and to be non-dense and they allow H\"{o}lder continuity of the resolvent.

\medskip
In this paper, we consider a ``light'' version of a non-autonomous Desch--Schappacher perturbation theorem. In particular, we ask for reasonable conditions which are easy to check in applications. Moreover, our approach, which relies on the works of Bertoni \cite{B2002,B2005,B2014}, provides a direct and easily accessible method to treat Desch--Schappacher perturbations in the non-autonomous setting.

\medskip
Another more indirect way to tackle non-autonomous problems is given by means of so-called evolution semigroups. These have been used for non-autonomous perturbations of bounded and Miyadera--Voigt type, see \cite{RRS1996} and \cite{RSRV2000}. However, we will not consider this method in the present paper.

\medskip
Let us outline the paper.
The first two sections have a preliminary character. We start by recalling the most important facts about autonomous abstract Cauchy problems as well as their associated extrapolation and Favard spaces. In Section \ref{sec:non-autonomous}, we revise non-autonomous abstract Cauchy problems and $\CD$-systems. In Section \ref{sec:non-autonomous_DS}, we state and prove the Desch--Schappacher perturbation result in our context, where we first consider extrapolation of $\CD$-systems and then we prove our main result Theorem \ref{thm:NonAutDS}. The final section consists of an example discussing non-autonomous uniformly strongly elliptic differential operators on $\mathrm{L}^p$-spaces, i.e.\ a so-called parabolic case.

\section{Autonomous Systems and $C_0$-Semigroups}
\label{sec:autonomous}

Let $X$ be a Banach space.

\subsection{$C_0$-semigroups}
The study of semigroups of operators started in the 1940's, see e.g.\ the monograph by Hille and Phillips \cite{HP1957}, and has since then attracted a lot of attention in the literature, see for example the manuscripts by Engel and Nagel \cite{EN}, Goldstein \cite{G2017} or Pazy \cite{P1983}, just to mention a few. 

\begin{definition}\label{def:OperatorSemigroup}
A family $(T(t))_{t\geq0}$ of bounded linear operators on $X$ is called a \emph{$C_0$-semigroup} if the following properties are satisfied:
\begin{iiv}
	\item $T(0)=\Id$, and $T(t+s)=T(t)T(s)$ for all $t,s\geq0$.
	\item $\lim_{t\to0}\left\|T(t)x-x\right\|=0$ for all $x\in X$.
\end{iiv}
\end{definition}

To each operator semigroup we can assign a closed operator $(A,\dom(A))$ called its (infinitesimal) generator.

\begin{definition}\label{def:Generator}
Let $(T(t))_{t\geq0}$ be a $C_0$-semigroup on $X$. The \emph{(infinitesimal) generator} $(A,\dom(A))$ of $(T(t))_{t\geq0}$ is defined by
\begin{align}\label{eqn:Generator}
Ax:=\lim_{t\to0}\frac{T(t)x-x}{t},\quad\dom(A):=\left\{x\in X:\ \lim_{t\to0}\frac{T(t)x-x}{t}\ \text{exists}\right\}.
\end{align}
\end{definition}

We notice that $C_0$-semigroups as well as their generators enjoy good properties, cf. \cite[Chapter II, Lemma~1.3]{EN}. For example, if $(A,\dom(A))$ is the generator of a $C_0$-semigroup $(T(t))_{t\geq0}$ then by \cite[Chapter II, Thm.~1.10]{EN} the resolvent $R(\lambda,A):=(\lambda-A)^{-1}$, $\lambda\in\rho(A)$, can be expressed by means of the Laplace transform whenever $\mathrm{Re}(\lambda)$ is sufficiently large, i.e., one has 
\begin{align}\label{eqn:LaplaceTrans}
R(\lambda,A)x=\int_0^\infty{\ee^{-\lambda t}T(t)x\ \dd{t}},\quad x\in X,
\end{align}
provided $\mathrm{Re}(\lambda)>\omega_0$, where $\omega_0$ denotes the growth bound of the semigroup $(T(t))_{t\geq0}$, cf.~\cite[Chapter I, Def.~5.6]{EN}. The celebrated Hille--Yosida theorem, cf. \cite[Chapter II, Thm.~3.8]{EN}, characterises generators of $C_0$-semigroups. Moreover, $C_0$-semigroups exactly provide the solution concept for autonomous abstract Cauchy problems
\begin{align}\label{eq:ACP}
 \begin{split}
    u'(t) & = Au(t),\quad t>0,\\
    u(0) & = x,
 \end{split}
\end{align}
where $A$ is a linear operator in $X$ and $x\in X$. Then \eqref{eq:ACP} is well-posed in the sense of Hadamard if and only if $A$ is the generator of a $C_0$-semigroup $(T(t))_{t\geq 0}$, and then $u:=T(\cdot)x$ yields the solution, see e.g.\ \cite[Corollary II.6.9]{EN}.

\subsection{Extrapolation spaces}
An important tool on our way towards a non-autonomous Desch--Schappacher theorem are the so-called extrapolation spaces. They are also important in several applications, see for example boundary perturbations \cite{Gr1987,HMR2015} or maximal regularity \cite{DaPG1984,NaSi2006}, just to mention a few. Let us recall the basic construction of extrapolation spaces, cf. \cite[Chapter II, Def.~5.4]{EN} or \cite[Def.~1.4]{NagelExtra}. Notice, that if $(A,\dom(A))$ is the generator of a $C_0$-semigroup then one may assume without loss of generality that $0\in\rho(A)$, i.e., $A$ is invertible with bounded inverse. 

\begin{definition}\label{def:Extrapolation}
Let $(A,\dom(A))$ be the generator of a $C_0$-semigroup $(T(t))_{t\geq0}$ on  $X$. With $\left\|\cdot\right\|_{-1}$ we denote the norm on $X$ defined by
\begin{align}\label{eqn:ExtrapolationNorm}
\left\|x\right\|_{-1}:=\left\|A^{-1}x\right\|,\quad x\in X.
\end{align} 
The completion of $(X,\left\|\cdot\right\|_{-1})$ is called the \emph{(first) extrapolation space} and is denoted by $X_{-1}$. If we want to emphasize that the extrapolation is with respect to the semigroup $(T(t))_{t\geq0}$ with generator $(A,\dom(A))$ we write $X_{-1}(T)$ or $X_{-1}(A)$, respectively. Furthermore, we denote the continuous extension of the operators $T(t)$ to the extrapolation space $X_{-1}$ by $T_{-1}(t)$. 
\end{definition}

The following result shows that in fact the family of operators $(T_{-1}(t))_{t\geq0}$ yields again a $C_0$-semigroup but now on the extrapolated space $X_{-1}$, cf. \cite[Chapter II, Thm.~5.5]{EN} or \cite[Thm.~1.5]{NagelExtra}.

\begin{proposition}\label{prop:ExtrapolationProperties}
\begin{abc}
	\item $X_{-1}$ is a Banach space containing $X$ as dense subspace.
	\item $(T_{-1}(t))_{t\geq0}$ is a $C_0$-semigroup on $X_{-1}$.
	\item The generator $A_{-1}$ of $(T_{-1}(t))_{t\geq0}$ has domain $\dom(A_{-1})=X$ and is the unique extension of $A:\dom(A)\to X$ to an isometry from $X$ to $X_{-1}$.
\end{abc}
\end{proposition}

Notice, that we only discussed the construction of the first extrapolation space since this is sufficient for our purpose. Furthermore, it follows from Proposition \ref{prop:ExtrapolationProperties} that $\left\|T_{-1}(t)\right\|=\left\|T(t)\right\|$ for all $t\geq0$, see also \cite{NS1994}. It is worth to mention that in general one can construct also extrapolation spaces of higher order, cf. \cite[Chapter II, Sect.~5a]{EN} or \cite[Sect.~1]{NagelExtra}.

\begin{remark}
Observe, that whenever $X$ is a separable Banach space, then $X_{-1}$ is also separable. This follows directly from the construction of extrapolation spaces.
\end{remark}

\subsection{The Favard class}
An important concept in the context of Desch--Schappacher perturbations is the so-called Favard class. 

\begin{definition}\label{def:Favardspace}
Let $(T(t))_{t\geq0}$ be a $C_0$-semigroup on $X$ with generator $(A,\dom(A))$ satisfying $\left[0,\infty\right)\subseteq\rho(A)$ and $\left\|\lambda R(\lambda,A)\right\|\leq M$ for all $\lambda>0$ and some constant $M\geq 1$. The \emph{Favard space} $F_1(A)$ corresponding to $A$ and $(T(t))_{t\geq0}$ is defined by
\begin{align}\label{eqn:Favardspace}
F_1(A):=\left\{x\in X:\ \sup_{t>0}{\frac{1}{t}\left\|T(t)x-x\right\|}<\infty\right\},
\end{align}
equipped with the norm $\left\|\cdot\right\|_{F_1(A)}$ defined by
\begin{align}\label{eqn:Favardnorm}
\left\|x\right\|_{F_1(A)}:=\sup_{t>0}{\frac{1}{t}\left\|T(t)x-x\right\|},\quad x\in F_1(A).
\end{align}
\end{definition}

By \cite[Chapter II, Thm.~5.15]{EN} the space $F_1(A)$ is a Banach space. If $X$ is reflexive then one has $F_1(A)=\dom(A)$. 

\begin{remark}\label{rem:Favardspace_interpolation}
    Let $(T(t))_{t\geq0}$ be a $C_0$-semigroup with generator $(A,\dom(A))$ satisfying $\left[0,\infty\right)\subseteq\rho(A)$ and $\left\|\lambda R(\lambda,A)\right\|\leq M$ for all $\lambda>0$ and some constant $M\geq 1$. Then
    $F_1(A) = (X,D(A))_{1,\infty,K}$ (with equivalence of norms), see \cite[Prop.~23]{B2005} or \cite[Prop.~5.7]{L2018}.
\end{remark}

By Proposition \ref{prop:ExtrapolationProperties} the extrapolated semigroup of a $C_0$-semigroup is a $C_0$-semigroup and hence we can construct the corresponding Favard space, which is then called the Favard class.

\begin{definition}\label{def:Favardclass}
Let $(T(t))_{t\geq0}$ be a $C_0$-semigroup on $X$ with generator $(A,\dom(A))$ satisfying $\left[0,\infty\right)\subseteq\rho(A)$ and $\left\|\lambda R(\lambda,A)\right\|\leq M$ for all $\lambda>0$ and some constant $M\geq 1$. The \emph{Favard class} $F_0(A)$ is defined to be the Favard space corresponding to the extrapolated semigroup $(T_{-1}(t))_{t\geq0}$ and its generator $A_{-1}$, i.e., 
\begin{align}\label{eqn:Favardclass}
F_0(A):=\left\{x\in X_{-1}:\ \sup_{t>0}{\frac{1}{t}\left\|T_{-1}(t)x-x\right\|_{-1}}<\infty\right\},
\end{align}
equipped with the norm $\left\|\cdot\right\|_{F_0(A)}$ defined by
\begin{align}\label{eqn:Favardclassnorm}
\left\|x\right\|_{F_0(A)}:=\sup_{t>0}{\frac{1}{t}\left\|T(t)x-x\right\|_{-1}},\quad x\in F_0(A).
\end{align}

\end{definition}

Note that $F_0(A)=F_1(A_{-1})$. In view of Proposition \ref{prop:ExtrapolationProperties}(c) and Remark \ref{rem:Favardspace_interpolation}, we have $F_0(A) = (X_{-1},X)_{1,\infty,K}$.

\section{Non-autonomous abstract Cauchy problems}
\label{sec:non-autonomous}

Let $X$ be a separable Banach space.

\subsection{Evolution families}
In what follows we consider, for fixed $T>0$ and a given family of (unbounded) operators $\bigl((A(t),\dom(A(t)))\bigr)_{t\in\left[0,T\right]}$ in $X$, the following non-autonomous abstract Cauchy problem

\begin{align}\tag{nACP}
\begin{split}
u'(t) & = A(t)u(t),\quad T\geq t\geq s\geq0,\\
u(s) & = x,
\end{split}
\end{align}
where $x\in X$. For the sake of completeness we recall what we mean by a (classical) solution of \eqref{eqn:nACP}.

\begin{definition}
Let $(A(t),\dom(A(t)))$, $t\in\left[0,T\right]$, be linear operators in $X$ and take $s\in\left[0,T\right]$ and $x\in\dom(A(s))$. Then a \emph{(classical) solution} of \eqref{eqn:nACP} is a function $u:=u(\cdot;s,x)\in\mathrm{C}^1\left(\left[s,T\right],X\right)$ such that $u(t)\in\dom(A(t))$ and $u$ satisfies \eqref{eqn:nACP} for $t\geq s$. The non-autonomous abstract Cauchy problem \eqref{eqn:nACP} is called \emph{well-posed (on spaces $Y_t$)} if there are dense subspaces $Y_s\subseteq\dom(A(s))$, $s\in\left[0,T\right]$, of $X$ such that for $s\in\left[0,T\right]$ and $x\in Y_s$ there is a unique solution $t\mapsto u(t;s,x)\in Y_t$ of \eqref{eqn:nACP}. In addition, for $s_n\to s$ and $Y_{s_n}\ni x_n\to x$, we have $\widetilde{u}(t;s_n,x_n)\to\widetilde{u}(t;s,x)$ uniformly for $t$ in compact intervals in $\RR$, where we set 
\[\widetilde{u}(t;s,x):=\begin{cases} {u}(t;s,x) & T\geq t\geq s\geq0,\\ x & t<s,\\ {u}(T;s,x) & t>T.\end{cases}\]
\end{definition}

\begin{remark}
    Note that there are different notions of solutions in the literature. For example, \cite{GalleratiVeraar2017} use so-called strong solutions with weaker regularity properties.
\end{remark}

The solution concept of strongly continuous semigroups of linear operators, in the autonomous case, is now replaced by evolution families in the non-autonomous situation. 

Let $\Delta_T:=\left\{(t,s)\in\left[0,T\right]^2:\ t\geq s\right\}$.

\begin{definition}
A family of bounded linear operators $(U(t,s))_{t,s\in\left[0,T\right], t\geq s}$ on $X$ is called a \emph{(strongly continuous) evolution family} if
\begin{iiv}
	\item $U(s,s)=\Id$ and $U(t,r)=U(t,s)U(s,r)$ for all $t,s,r\in\left[0,T\right]$ with $T\geq t\geq s\geq r\geq0$, and
	\item the mapping $\Delta_T\ni(t,s)\mapsto U(t,s)$ is strongly continuous.
\end{iiv}
We say that $(U(t,s))_{t\geq s}$ \emph{solves the non-autonomous abstract Cauchy problem \eqref{eqn:nACP} (on spaces $Y_t$)} if there are dense subspaces $Y_s$, $s\in\left[0,T\right]$, of $X$ such that $U(t,s)Y_s\subseteq Y_t\subseteq\dom(A(t))$ for $T\geq t\geq s\geq0$ and the function $t\mapsto U(t,s)x$ is a solution of \eqref{eqn:nACP} for $s\in\left[0,T\right]$ and $x\in Y_s$.
\end{definition} 

That evolution families are the right concept in the non-autonomous situation is justified by the following fact.

\begin{proposition}
The non-autonomous abstract Cauchy problem \eqref{eqn:nACP} is well-posed on $Y_t$ if and only if there is an evolution family solving \eqref{eqn:nACP} on $Y_t$.
\end{proposition}

Here and in the sequel we assume that \eqref{eqn:nACP} is well-posed in the sense of \cite[Chapter VI, Def. 9.1]{EN}.
\, Especially, there exists an evolution family $(U(t,s))_{t\geq s}$ solving \eqref{eqn:nACP}. Observe that some authors also call $(U(t,s))_{t\geq s}$ a \emph{parabolic fundamental solution}. The existence of such a fundamental solution is a priori not clear and is treated by Acquistapace and Terreni \cite{Acquistapace1987,Acquistapace1988} or Kato and Tanabe \cite{Tanabe1960,KT1962,Tanabe1961,Kato1961} and many others. 

\subsection{$\CD$-systems} 
The main ingredient for the non-autonomous Desch--Schappacher perturbation theory will be the notion of $\CD$-systems. Before introducing such systems, we recall the notion of Kato stability which will be important for our purpose.

\begin{definition}[Kato stability]\label{def:KatoStab}
Let $T>0$. A family of linear operators $\bigl((A(t),\dom(A(t)))\bigr)_{t\in\left[0,T\right]}$ of $C_0$-semigroup generators on a Banach space $X$ is called \emph{stable} if there exist $M\geq1$ and $\omega\in\RR$ such that $\left(\omega,\infty\right)\subseteq\rho(A(t))$ for each $t\in\left[0,T\right]$ and
\begin{align}\label{eqn:KatoStableResolv}
\left\|\prod_{j=1}^k{R(\lambda,A(t_j))}\right\|=\left\|R(\lambda,A(t_k)))\cdots R(\lambda,A(t_1))\right\|\leq\frac{M}{(\lambda-\omega)^k},
\end{align}
for all $\lambda>\omega$ and every partition $0\leq t_1\leq t_2\leq\ldots\leq t_k\leq T$, $k\in\NN$.
\end{definition}

\begin{remark}
It has been shown by Kato \cite{K1970} that \eqref{eqn:KatoStableResolv} is equivalent to the existence of $M\geq 1$ and $\omega\in\RR$ such that 
\begin{align}\label{eqn:KatoStableSemi}
\left\|\prod_{j=1}^k{\ee^{\tau_jA(s_j)}}\right\|=\left\|\ee^{\tau_kA(s_k)}\cdots\ee^{\tau_1A(s_1)}\right\|\leq M\ee^{\omega\sum_{j=1}^k{s_j}},
\end{align}
for all partitions $0\leq s_1\leq s_2\leq\ldots\leq s_k\leq T$ and all $\tau_k\geq0$, $k\in\NN$. Here, $(\ee^{\tau A(s)})_{\tau\geq0}$ denotes the $C_0$-semigroup generated by the operator $(A(s),\dom(A(s)))$ for fixed $s\in\left[0,T\right]$
\end{remark}

\begin{remark}\label{rem:KatoStabQuasiCont}
Note that $X_{-1}$ is separable since $X$ is separable.
As already mentioned by Pazy \cite[Sect.~5.2]{P1983}, the order of the resolvent operators in \eqref{eqn:KatoStableResolv} is important since in general the resolvent operators do not commute. There it is also noticed, that if each operator $(A(t),\dom(A(t)))$, $t\in\left[0,T\right]$ is the generator of a quasi-contractive $C_0$-semigroup with a uniform growth bound then the family of operators $\bigl((A(t),\dom(A(t)))\bigr)_{t\in\left[0,T\right]}$ is stable in the sense of Definition \ref{def:KatoStab}. In particular, any family $\bigl((A(t),\dom(A(t)))\bigr)_{t\in\left[0,T\right]}$ of generators of contractive $C_0$-semigroups is stable.
\end{remark}

Let us now come to the definition of $\CD$-systems.

\begin{definition}[$\CD$-system]\label{def:CDsystem}
Let $T>0$, $\bigl((A(t),\dom(A(t)))\bigr)_{t\in\left[0,T\right]}$ a family of generators of $C_0$-semigroups on $X$. Let $\dom$ be a Banach space. Then the triplet $((A(t))_{t\in [0,T]},X,\dom)$ is called a \emph{$\CD$-system} if the following properties are satisfied:
\begin{abc}
	\item The domain $\dom(A(t))=\dom$ is independent of $t\in\left[0,T\right]$, where $\dom$ is continuously and densely embedded in $X$.
	\item The family $\bigl((A(t),\dom)\bigr)_{t\in\left[0,T\right]}$ is stable in the sense of Definition \ref{def:KatoStab}.
	\item $A(\cdot)\in\Lip\left(\left[0,T\right],\LLL_{\mathrm{s}}(\dom,X)\right)$, i.e., $A(\cdot)$ is strongly Lipschitz continuous.
\end{abc}
\end{definition}

As a matter of fact, $\CD$-systems automatically yield solutions of the associated non-autonomous abstract Cauchy problem \eqref{eqn:nACP}, see e.g.\ \cite{K1970,K1973} or \cite[Sect.~1.2]{K1985}.

\begin{theorem}\label{thm:ExSolCD}
Let $T>0$, $((A(t))_{t\in[0,T]},X,\dom)$ a $\CD$-system. Then there exists a unique evolution family $(U(t,s))_{0\leq s\leq t\leq T}$ with the following properties:
\begin{iiv}
	\item $U(t,r)U(r,s)=U(t,s)$ and $U(s,s)=\Id$ for all $t,r,s\in\left[0,T\right]$ with $t\geq r\geq s$.
	\item $U(\cdot,\cdot)\in\mathrm{C}\left(\Delta_T,\LLL_\mathrm{s}(X)\right)\cap\mathrm{C}\left(\Delta_T,\LLL_\mathrm{s}(\dom)\right)$.
	\item $\frac{\partial}{\partial{t}}U(t,s)x=A(t)U(t,s)x$ for all $x\in\dom$ and $t,s\in [0,T]$ with $t\geq s$.
	\item $\frac{\partial}{\partial{s}}U(t,s)x=-U(t,s)A(s)x$ for all $x\in\dom$ and $t,s\in [0,T]$ with $t\geq s$.
	\item There exists $M\geq1$ and $\omega\in\RR$ such that $\left\|U(t,s)\right\|\leq M\ee^{\omega(t-s)}$ for all $t,s\in\left[0,T\right]$ with $t\geq s$.
\end{iiv}
\end{theorem}

\begin{example}\label{exa:Bertoni}
    Let $T>0$ and $g,\mu\from [0,T]\times \RR_{\geq 0}\to \RR_{>0}$ such that $\mu$ is bounded away from zero, $g(t,\cdot)\in \mathrm{W}^{1,\infty}(\RR_{\geq0})$ and $\frac{1}{g(t,\cdot)}, \mu(t,\cdot)\in \mathrm{L}^\infty(\RR_{\geq0})$ for all $t\in[0,T]$,
    and $t\mapsto g(t,x)$, $t\mapsto \frac{\partial}{\partial x} g(t,x)$ and $t\mapsto\mu(t,x)$ are Lipschitz continuous uniformly in a.e.\ $x\in \RR_{\geq 0}$.
    Let $X:=\mathrm{L}^1\left(\RR_{\geq0}\right)$ and define $\bigl((A(t),\dom(A(t)))\bigr)_{t\in[0,T]}$ by
\begin{align}
(A(t)f)(x):=-\frac{\partial}{\partial{x}}g(t,x)f(x)-\mu(t,x)f(x),\quad t\in\left[0,T\right],
\end{align}
on the constant domain
\[
\dom:=\dom(A(t)):=\left\{f\in\mathrm{W}^{1,1}\left(\RR_{\geq0}\right):\ f(0)=0\right\}.
\]
Then \cite[Prop.~19]{B2005} shows that $((A(t))_{t\in[0,T]},X,\dom)$ is a $\CD$-system.
\end{example}

\section{Non-autonomous Desch--Schappacher perturbations}
\label{sec:non-autonomous_DS}

Let $X$ be a separable Banach space and $T>0$.

\subsection{Perturbations of $\CD$-systems}
The following result by Bertoni \cite[Prop.~6]{B2014} regarding perturbations of $\CD$-systems will be crucial for our non-autonomous Desch--Schappacher perturbation result. Note that given a $\CD$-system $((A(t))_{t\in [0,T]},X,\dom)$ we have that 
\[F_1(A(t)) = (X;\dom)_{1,\infty,K} = F_1(A(s)),\quad 0\leq s,t\leq T,\]
so the Favard spaces are constant. Let $F_1:=F_1(A(0)) = F_1(A(t))$ for all $t\in[0,T]$.

\begin{proposition}\label{prop:PertCDSyst}
Let $((A(t))_{t\in [0,T]},X,\dom)$ be a $\CD$-system. If
\[
B(\cdot)\in\mathrm{L}^{\infty}\left(\left[0,T\right],\LLL_{\mathrm{s}}(\dom,F_1)\right)\cap\Lip\left(\left[0,T\right],\LLL_{\mathrm{s}}(\dom,X)\right)
\]
then $((A(t)+B(t))_{t\in[0,T]},X,\dom)$ is again a $\CD$-system.
\end{proposition}

In other words, Proposition \ref{prop:PertCDSyst} tells us that whenever we perturb a $\CD$-system with a sufficiently regular family of operators then we obtain again a $\CD$-system and hence by Theorem \ref{thm:ExSolCD} an evolution family solving the associated \eqref{eqn:nACP}.

\subsection{$\CD$-systems in extrapolation spaces}

Let $\bigl((A(t),\dom(A(t)))\bigr)_{t\in[0,T]}$ be a family of closed and densely defined operators in $X$. Under suitable assumptions we are going to show that the family of operators $((A(t)_{-1})_{t\in [0,T]},X_{-1},X)$ is $\CD$-system. To get this result, we first state and comment the on the assumptions we use.

\begin{assumptionA}\label{assump:GenAnalySemi}
There exist $\theta\in\left(\frac{\pi}{2},\pi\right)$ and $M\geq 1$ such that
\begin{iiv}
	\item $\Sigma_{\theta,0} \subseteq\rho(A(t))$ for all $t\in\left[0,T\right]$ and
	\item $\left\|R(\lambda,A(t))\right\|\leq\frac{M}{\left|\lambda\right|}$ for all $\lambda\in\Sigma_{\theta,0}\setminus\left\{0\right\}$ and $\left\|A(t)^{-1}\right\|\leq M$ for all $t\in\left[0,T\right]$,
\end{iiv}
where $\Sigma_{\theta,0}:=\left\{z\in\mathbb{C}:\ \left|\arg(z)\right|\leq\theta\right\}\cup\left\{0\right\}$ is a sector in the complex plane.
\end{assumptionA}

\begin{remark}\label{rem:GenAnalySemi}
We observe that Assumption \ref{assump:GenAnalySemi} guarantees that each operator $(A(t),\dom(A(t)))$, $t\in\left[0,T\right]$ generates a strongly continuous and analytic semigroup on the Banach space $X$. In particular, one can construct the extrapolation spaces $X_{-1}(A(t))$ and extrapolated operators $A_{-1}(t):=A(t)_{-1}$ for each $t\in\left[0,T\right]$. We notice that by construction of extrapolation spaces all operators $A_{-1}(t)$ have the same domain, namely $\dom(A_{-1}(t))=X$ for all $t\in\left[0,T\right]$. Furthermore, by \cite[Chapter V, Prop.~1.3.1]{Ama1995}, it holds that $\rho(A(t))=\rho(A_{-1}(t))$ for all $t\in\left[0,T\right]$.
\end{remark}

\begin{lemma}\label{lemma:EquivLip}
    Let $\bigl((A(t),\dom(A(t)))\bigr)_{t\in[0,T]}$ satisfy Assumption \ref{assump:GenAnalySemi} and assume that $\dom(A(t))=\dom(A(s))$ for all $s,t\in[0,T]$. 
    \begin{abc}
    \item The following are equivalent.
    \begin{iiv}
        \item There exists $C\geq0$ such that
        $\left\|\ee^{\tau A(t)}-\ee^{\tau A(s)}\right\|\leq C\left|t-s\right|$
        for all $t,s\in[0,T]$, $\tau\geq 0$.
        \item There exists $C\geq0$ such that
        $\|R(\lambda,A(t))-R(\lambda,A(s))\|\leq\frac{C}{\left|\lambda\right|}\left|t-s\right|$ for all $t,s\in[0,T]$, $\lambda\in \Sigma_{\theta,0}\setminus\{0\}$.
    \end{iiv}
    \item
        Assume there exists $C\geq0$ such that
        $\left\|\Id-A(t)A(s)^{-1}\right\|\leq C\left|t-s\right|$ for all $t,s\in[0,T]$.
        Then there exists $C\geq0$ such that
        $$\left\|R(\lambda,A(t))-R(\lambda,A(s))\right\|\leq\frac{C}{\left|\lambda\right|}\left|t-s\right|$$ for all $t,s\in[0,T]$, $\lambda\in \Sigma_{\theta,0}\setminus\{0\}$.
    \end{abc}
\end{lemma}

\begin{proof}
\begin{abc}
\item
The implication (i)$\Rightarrow$(ii) follows directly from the fact that the resolvent can be expressed in terms of the Laplace transformation \eqref{eqn:LaplaceTrans} (at first for $\lambda>0$ which then implies the statement for all $\lambda\in \Sigma_{\theta,0}\setminus\{0\}$). The converse (ii)$\Rightarrow$(i) is \cite[Cor.~2.6]{Sch1999} and uses that
\[
\ee^{\tau A(s)}=\frac{1}{2\pi i}\int_\Gamma{\ee^{\lambda\tau}R(\lambda,A(s))\ \dd\lambda},
\]
see also \cite[Chapter II, Def.~4.2]{EN}. 
\item
This follows from \cite[Lemma~2.2]{AT1985} by observing that
\[
R(\lambda,A(t))-R(\lambda,A(s))=R(\lambda,A(t))(A(t)A(s)^{-1}-\Id)A(s)R(\lambda,A(s)). \qedhere
\]
\end{abc}
\end{proof}

\begin{assumptionA}\label{assump:ConstExtra}
Let $\bigl((A(t),\dom(A(t)))\bigr)_{t\in\RR}$ satisfy Assumption \ref{assump:GenAnalySemi} and denote by $X_{-1}(A(t))$, $t\in\left[0,T\right]$, the extrapolation space with respect to $(A(t),\dom(A(t)))$, $t\in\left[0,T\right]$. We assume that 
\begin{align*}
X_{-1}(A(t))\cong X_{-1}(A(0))=:X_{-1},
\end{align*}
for all $t\in\left[0,T\right]$ and such that there exists $\kappa>0$ with
\begin{align}\label{eqn:eqivNorm}
\frac{1}{\kappa}\left\|x\right\|_{X_{-1}}\leq\left\|x\right\|_{X_{-1}(A(t))}\leq\kappa\left\|x\right\|_{X_{-1}},\quad x\in X_{-1},\ t\in\left[0,T\right].
\end{align}
\end{assumptionA}

\begin{remark}\label{rem:ConstExtra}
Assumption \ref{assump:ConstExtra} appears to be natural in the context of extrapolation of evolution families. The problem is that, in general, for $s,t\in\left[0,T\right]$, $s\neq t$ the corresponding extrapolation spaces $X_{-1}(A(s))$ and $X_{-1}(A(t))$ may be completely different. However, Assumption \ref{assump:ConstExtra} allows to extrapolate a given evolution family, cf. \cite[Thm.~7.2]{Ama1988}. Moreover, it allows us to transport the uniform resolvent estimate from Assumption \ref{assump:GenAnalySemi} to the extrapolation space, i.e., we have that
\[
\left\|R(\lambda,A_{-1}(t))\right\|\leq\frac{M}{\left|\lambda\right|},\quad \lambda\in\Sigma_{\theta,0}\setminus\left\{0\right\}\quad \!\text{and}\!\quad \left\|A_{-1}(t)^{-1}\right\|\leq M,\quad t\in\left[0,T\right].
\]
To see this, observe that for all $x\in X$ we have by \eqref{eqn:eqivNorm} that
\begin{align*}
& \left\|R(\lambda,A_{-1}(t))x\right\|_{X_{-1}}\leq\kappa\left\|R(\lambda,A_{-1}(t))x\right\|_{X_{-1}(A(t))}=\kappa\left\|R(\lambda,A(t))A(t)^{-1}x\right\|\\
&\leq\left\|R(\lambda,A(t))\right\|\cdot\kappa\left\|x\right\|_{X_{-1}(A(t))}\leq\left\|R(\lambda,A(t))\right\|\cdot\kappa^2\left\|x\right\|_{X_{-1}}.
\end{align*}
Thus, we observe that by combining Assumptions \ref{assump:GenAnalySemi} and \ref{assump:ConstExtra} we observe that Assumption \ref{assump:GenAnalySemi} also holds for the extrapolated family $\bigl((A_{-1}(t),X)\bigr)_{t\in [0,T]}$ (see Remark \ref{rem:GenAnalySemi} for (i) and Remark \ref{rem:ConstExtra} for (ii)). In particular, each operator $(A_{-1}(t),X)$, $t\in\left[0,T\right]$, generates a $C_0$-semigroup. 
\end{remark}

\begin{example}
It was shown by Bertoni \cite[Rem.~9]{B2005} that the $\CD$-system we mentioned in Example \ref{exa:Bertoni} has constant extrapolation spaces and that $X_{-1}=\left\{\varphi\in\mathscr{D}':\ \exists g\in\mathrm{L}^1(\RR_{\geq 0}): \varphi=g'\right\}$.
\end{example}

In view of Remark \ref{rem:GenAnalySemi} together with Remark \ref{rem:ConstExtra}, one might suspect that $((A_{-1}(t))_{t\in[0,T]},X_{-1},X)$ is a $\CD$-system. In order to show this, we need another additional assumption which is inspired by \cite[Hypo.~II]{AT1985}.

\begin{assumptionA}\label{assump:Lipschitz}
There exists $L>0$ such that $\left\|\Id-A_{-1}(t)A_{-1}(s)^{-1}\right\|\leq L\left|t-s\right|$ for all $t,s\in\left[0,T\right]$.
\end{assumptionA}

\begin{remark}
Observe that in assumption \cite[Hypo.~II]{AT1985}, Acquistapace and Terreni make use of a slightly weaker regularity assumption than Assumption \ref{assump:Lipschitz}. In fact, they assume H\"{o}lder instead of Lipschitz continuity, i.e., there exist $L>0$ and $\alpha\in\left(0,1\right)$ such that $\left\|\Id-A_{-1}(t)A_{-1}(s)^{-1}\right\|\leq L\left|t-s\right|^\alpha$ for all $t,s\in\left[0,T\right]$.
\end{remark}

We can now show that $((A_{-1}(t))_{t\in[0,T]},X_{-1},X)$ is a $\CD$-system.

\begin{theorem}\label{thm:CDsystemsExtra}
Let $\bigl((A(t),\dom(A(t)))\bigr)_{t\in[0,T]}$ satisfy Assumptions \ref{assump:GenAnalySemi}--\ref{assump:Lipschitz}. Then $((A_{-1}(t))_{t\in[0,T]},X_{-1},X)$ is a $\CD$-system.
\end{theorem}

\begin{proof}
Note that $X_{-1}$ is separable since $X$ is separable. First of all, we show that $\bigl((A_{-1}(t),X)\bigr)_{t\in[0,T]}$ is stable in the sense of Definition \ref{def:KatoStab}. By Remark \ref{rem:GenAnalySemi} and Remark \ref{rem:ConstExtra}, we observe that the family of operators $\bigl((A_{-1}(t),X)\bigr)_{t\in[0,T]}$ has constant domain and satisfies Assumption \ref{assump:GenAnalySemi}. From Lemma \ref{lemma:EquivLip} we conclude that there exists $C\geq 0$ such that
\begin{align}\label{eqn:ResolventLipschitz}
\left\|R(\lambda,A_{-1}(t))-R(\lambda,A_{-1}(s))\right\|\leq\frac{C}{\left|\lambda\right|}\left|t-s\right|,
\end{align}
for all $s,t\in\left[0,T\right]$ and each $\lambda\in\Sigma_{\theta,0}\setminus\left\{0\right\}$ or equivalently that there exists $C\geq 0$ such that
\[
\left\|\ee^{\tau A_{-1}(t)}-\ee^{\tau A_{-1}(s)}\right\|\leq C\left|t-s\right|,
\]
for all $s,t\in\left[0,T\right]$ and $\tau\geq 0$, which by \cite[Prop.~4.9]{Sch1996} implies that the family of operators $\bigl((A_{-1}(t),X)\bigr)_{t\in\left[0,T\right]}$ is stable. In order to show that the family $((A_{-1}(t))_{t\in[0,T]},X_{-1},X)$ is a $\CD$-system, it remains to show that $A_{-1}(\cdot)$ is strongly Lipschitz continuous, which directly follows from the Lipschitz continuity of the resolvents, see \eqref{eqn:ResolventLipschitz}, and \cite[Thm.~3(b)]{B2002}.
\end{proof}

\begin{remark}\label{rem:ExtraEvolFam}
Theorem \ref{thm:CDsystemsExtra} yields that the Assumptions \ref{assump:GenAnalySemi}--\ref{assump:Lipschitz} imply that $((A_{-1}(t))_{t\in[0,T]},X_{-1},X)$ is a $\CD$-system and hence by Theorem \ref{thm:ExSolCD} we obtain a unique evolution family $(\widetilde{U}(t,s))_{t\geq s}$ corresponding to the family of operators $\bigl((A_{-1}(t),X)\bigr)_{t\in\left[0,T\right]}$. Moreover, by \cite[Thm.~7.2]{Ama1988}, there exists an evolution family $(U_{-1}(t,s))_{t\geq s}$ associated with $\bigl((A_{-1}(t),X)\bigr)_{t\in\left[0,T\right]}$ such that $U_{-1}(t,s)_{|X}=U(t,s)$ (therefore also called extrapolated evolution family). Thus, uniqueness implies $\widetilde{U}(t,s)=U_{-1}(t,s)$ for all $t,s\in\left[0,T\right]$ with $t\geq s$.
\end{remark}

\begin{remark}\label{rem:LipschitzAssump}
By following the proof of Theorem \ref{thm:CDsystemsExtra} it suffices that \eqref{eqn:ResolventLipschitz} holds true (which by Lemma \ref{lemma:EquivLip} follows from Assumption \ref{assump:Lipschitz}). Since also Assumption \ref{assump:ConstExtra} holds and $R(\lambda,A_{-1}(t))_{|X}=R(\lambda,A(t))$ we obtain that \eqref{eqn:ResolventLipschitz} is already satisfied if we assume that the resolvents of $(A(t))_{t\in [0,T]}$ on $X$ are commuting and Lipschitz continuous; that is, 
\[R(\lambda,A(t))R(\mu,A(s)) = R(\mu,A(s))R(\lambda,A(t)),\quad  s,t\in [0,T],\, \lambda,\mu\in \Sigma_{\theta,0},\]

and that there exists $C'\geq 0$ such that
\begin{align}\label{eqn:ResolventLipschitz1}
\left\|R(\lambda,A(t))-R(\lambda,A(s))\right\|\leq\frac{C'}{\left|\lambda\right|}\left|t-s\right|,
\end{align}
for all $s,t\in\left[0,T\right]$ and each $\lambda\in\Sigma_{\theta,0}\setminus\left\{0\right\}$.
Indeed, we then observe for $x\in X$ that
\begin{align*}
    \norm{\bigl(R(\lambda,A_{-1}(t))-R(\lambda,A_{-1}(s))\bigr)x}_{X_{-1}} 
    & \leq  \kappa \norm{A(t)^{-1}\bigl(R(\lambda,A(t))-R(\lambda,A(s))\bigr)x}_{X} \\
    & = \kappa \norm{\bigl(R(\lambda,A(t))-R(\lambda,A(s))\bigr)A(t)^{-1}x}_{X} \\
    & \leq \kappa^2\frac{C'}{\abs{\lambda}}\abs{t-s} \norm{x}_{X_{-1}}
\end{align*}
for all $s,t\in [0,T]$ and $\lambda\in \Sigma_{\theta,0}\setminus\{0\}$, which implies \eqref{eqn:ResolventLipschitz} by density of $X$ in $X_{-1}$.
Note that commuting resolvents, sometimes used as definition of commutation of unbounded operators, have also been used in other contexts \cite{DoreVenni1987, MonniauxPruess1997}.

In general, condition \eqref{eqn:ResolventLipschitz1} can be checked easier than Assumption \ref{assump:Lipschitz} due to the fact that one does not need the knowledge of the extrapolation spaces and the extrapolated operators.
\end{remark} 

\subsection{A non-autonomous Desch--Schappacher perturbation}

We are now able to state and prove our main result, a non-autonomous Desch--Schappacher perturbation type result. Let $\bigl((A(t),\dom(A(t)))\bigr)_{t\in[0,T]}$ be a family of closed and densely defined operators in $X$ satisfying Assumption \ref{assump:GenAnalySemi}. 
Due to the fact that $\dom(A_{-1}(t))=X$ for all $t\in\left[0,T\right]$ we conclude that the Favard classes $F_0(A(t))$ are also constant in time; that is, $F_0:=F_0(A(0)) = F_0(A(t))$ for all $t\in[0,T]$, since $F_0(A(t))=(X_{-1},X)_{1,\infty,K} =F_0(A(0))$, see the last sentence of Section \ref{sec:autonomous}. 

\begin{theorem}\label{thm:NonAutDS}
Let $\bigl((A(t),\dom(A(t)))\bigr)_{t\in\left[0,T\right]}$ satisfy Assumption \ref{assump:GenAnalySemi}--\ref{assump:Lipschitz} and assume that the corresponding \eqref{eqn:nACP} yields a solution $(U(t,s))_{t\geq s}$. If 
\[
B(\cdot)\in\mathrm{L}^{\infty}\left(\left[0,T\right],\LLL_{\mathrm{s}}(X,F_0)\right)\cap\Lip\left(\left[0,T\right],\LLL_{\mathrm{s}}(X,X_{-1})\right),
\]
then there exists an evolution family $(V(t,s))_{t\geq s}$ on $X$ satisfying the variation of constant formula
\[
{V}(t,s)x=U(t,s)x+\int_s^t{U_{-1}(t,\sigma)B(\sigma){V}(\sigma,s)x\ \dd{\sigma}},\quad t\geq s, x\in X.
\]
\end{theorem}

\begin{proof}
By Theorem \ref{thm:CDsystemsExtra} we know that $((A_{-1}(t))_{t\in[0,T]},X_{-1},X)$ is a $\CD$-system, yielding the solution $(U_{-1}(t,s))_{t\geq s}$, see Remark \ref{rem:ExtraEvolFam}. Due to the assumed regularity of $B(\cdot)$ we conclude by Proposition \ref{prop:PertCDSyst} that $((A_{-1}(t)+B(t))_{t\in[0,T]},X_{-1},X)$ is also a $\CD$-system. Therefore, Theorem \ref{thm:ExSolCD} ensures the existence of a unique solution of the corresponding \eqref{eqn:nACP} which we call $(\widetilde{V}(t,s))_{t\geq s}$. This family of operators satisfies the variation of constant formula on $X_{-1}$, i.e.,
\[
\widetilde{V}(t,s)x=U_{-1}(t,s)x+\int_s^t{U_{-1}(t,\sigma)B(\sigma)\widetilde{V}(\sigma,s)x\ \dd{\sigma}},\quad t\geq s,\ x\in X_{-1}.
\]
By Theorem \ref{thm:ExSolCD} we have that $\widetilde{V}\in\mathrm{C}(\Delta_T,\LLL_\mathrm{s}(X_{-1},X_{-1}))\cap\mathrm{C}(\Delta_T,\LLL_\mathrm{s}(X,X))$. Thus, we know that the restriction $(V(t,s))_{t\geq s}$ of $(\widetilde{V}(t,s))_{t\geq s}$ to $X$ yields a continuous family which then satisfies
\[
{V}(t,s)x=U(t,s)x+\int_s^t{U_{-1}(t,\sigma)B(\sigma){V}(\sigma,s)x\ \dd{\sigma}},\quad t\geq s,\ x\in X. \qedhere
\]
\end{proof}


\begin{remark}
    Let us relate Theorem \ref{thm:NonAutDS} to the existing literature. 
    \begin{abc}
        \item
        In \cite[Thm.~2.3]{Rhandi1997}, non-autonomous perturbations of autonomous abstract Cauchy problems were considered, i.e.\ there exists a generator $A_0$ in $X$ such that $A(t)=A_0$ for all $t\in [0,T]$, and $B$ was only assumed to be continuous (with values in $\LLL_{\mathrm{s}}(X,F_0(A))$). We need Lipschitz continuity (with values in $\LLL_{\mathrm{s}}(X,X_1)$) due to the fact that we have non-autonomous $A(\cdot)$.
        \item
        In \cite[Cor.~3.7]{Xiao2002}, also non-autonomous perturbations of autonomous abstract Cauchy problems were considered, i.e.\ there exists a generator $A_0$ in $X$ such that $A(t)=A_0$ for all $t\in [0,T]$, where $B$ also only needs to be continuous (with values in $\LLL_{\mathrm{s}}(X)$) and bounded (with values in $\LLL_{\mathrm{s}}(\dom(A_0),F_1(A_0))$).
        We need a corresponding boundedness assumption on $B$ (with values in $\LLL_{\mathrm{s}}(X,F_0(A_0))$, which is the same assumption as in \cite{Xiao2002} considered in the extrapolation space), but Lipschitz continuity (with values in $\LLL_{\mathrm{s}}(X,X_1)$).
    \end{abc}
\end{remark}

\begin{remark}
It is worth to mention that even if \eqref{eqn:nACP} is well-posed the perturbed non-autonomous abstract Cauchy problem does not have to be well-posed in general. In fact, the evolution family $(V(t,s))_{t\geq s}$ one obtains from Theorem \ref{thm:NonAutDS} can be interpreted as mild solution of the corresponding Cauchy problem. We also refer to \cite[Chapter VI, Ex.~9.21]{EN}.
\end{remark}

In view of Remark \ref{rem:LipschitzAssump}, we can replace Assumption \ref{assump:Lipschitz} by \eqref{eqn:ResolventLipschitz1} in Theorems \ref{thm:CDsystemsExtra} and \ref{thm:NonAutDS} if the resolvents of $(A(t))_{t\in [0,T]}$ are commuting. We formulate this as a corollary.

\begin{corollary}\label{cor:DS_extra_withoutC}
    Let $\bigl((A(t),\dom(A(t)))\bigr)_{t\in\left[0,T\right]}$ satisfy Assumption \ref{assump:GenAnalySemi}--\ref{assump:ConstExtra} and let the resolvents of $(A(t))_{t\in[0,T]}$ be commuting and satisfy \eqref{eqn:ResolventLipschitz1}. Then $((A_{-1}(t))_{t\in[0,T]},X_{-1},X)$ is a $\CD$-system.  
    Moreover, if we assume that the corresponding \eqref{eqn:nACP} yields a solution $(U(t,s))_{t\geq s}$ and if
    \[
    B(\cdot)\in\mathrm{L}^{\infty}\left(\left[0,T\right],\LLL_{\mathrm{s}}(X,F_0)\right)\cap\Lip\left(\left[0,T\right],\LLL_{\mathrm{s}}(X,X_{-1})\right),
    \]
    then there exists an evolution family $(V(t,s))_{t\geq s}$ on $X$ satisfying the variation of constant formula
    \[
    V(t,s)x=U(t,s)x+\int_s^tU_{-1}(t,\sigma)B(\sigma){V}(\sigma,s)x\ \dd\sigma,\quad t\geq s, x\in X.
    \]
\end{corollary}

\section{Example: Non-autonomous uniformly strongly elliptic differential operators}

Let $T>0$. Let $m\in\NN$, and for $\alpha\in\NN_0^d$ with $\abs{\alpha}\leq m$ let $a_\alpha\from[0,T]\to\CC$ be Lipschitz continuous. Let $a\from [0,T]\times \RR^d\to\CC^d$ be defined by
\[a(t,\xi):=\sum_{\abs{\alpha} \leq m} a_\alpha(t) (i\xi)^\alpha.\]
We assume that $a$ is uniformly strongly elliptic, i.e.\ for the principal part $a_m\from [0,T]\times \RR^d\to\CC^d$, $a_m(t,\xi):= \sum_{\abs{\alpha} = m} a_\alpha(t) (i\xi)^\alpha$ of $a$ there exists $c>0$ such that
\[\Re a_m(t,\xi) \geq c\abs{\xi}^m,\quad t\in[0,T], \xi\in\RR^d.\]

Additionally, let us further assume that there exists $\omega>0$ such that $\Re a(t,\xi)\geq \omega$ for all $t\in[0,T]$ and $\xi\in\RR^d$ (which we can achieve by adding a suitable constant to $a_0$).

\medskip
Let $p\in(1,\infty)$. Then $a$ gives rise to a family $\bigl((A_p(t),\dom(A_p(t)))\bigr)_{t\in[0,T]}$ of negatives of m-sectorial differential operators in $X:=\mathrm{L}^p(\RR^d)$ as
\begin{align*}
    \dom(A_p(t)) & := \mathrm{W}^{p,m}(\RR^d),\\
    A_p(t) u & := -\sum_{\abs{\alpha}\leq m} a_\alpha(t) \partial^\alpha u
\end{align*}
for $t\in[0,T]$, see, e.g.,\ \cite[Chapter 8]{Haase2006}, \cite{BombachGabelSeifertTautenhahn2021}.
Note that by uniform strong ellipticity, $\norm{\cdot}_{\mathrm{W}^{p,m}(\RR^d)}$ and the graph norm of $A_p(s)$ are equivalent (with uniform constants for $s\in[0,T]$).

\medskip
By \cite[Theorem 8.2.1]{Haase2006}, Assumption \ref{assump:GenAnalySemi} is satisfied. Moreover, it is easy to see that the family of operators $\bigl((A_p(t),\mathrm{W}^{p,m}(\RR^d))\bigr)_{t\in[0,T]}$ satisfies Assumption \ref{assump:ConstExtra} and that $X_{-1}\cong\mathrm{W}^{p,m}(\RR^d)' = \mathrm{W}^{p',-m}(\RR^d)$, which follows from a duality argument and \cite[Chapter 3, 3.14]{AdamsFournier2003}. Here, $p'\in(1,\infty)$ is the H\"older dual exponent, i.e.\ $\frac{1}{p} + \frac{1}{p'} = 1$.

\medskip
Moreover, since $A_p(s)$ and $A_p(t)$ commute on $W^{p,2m}(\RR^d)$ for all $s,t\in [0,T]$ (as the coefficients are only depending on time and not on space), we observe that the resolvents of $(A_p(t))_{t\in[0,T]}$ commute. In order to show that \eqref{eqn:ResolventLipschitz1} is satisfied, we make use of Lemma \ref{lemma:EquivLip}(b). Let $f\in \mathrm{L}^p(\RR^d)$, $u(t):=A_p(t)^{-1}f$ for $t\in[0,T]$. By the Lipschitz continuity of the coefficient functions $a_\alpha$ there exists $C\geq 0$ such that
\[\norm{(\Id - A_p(t)A_p(s)^{-1})f} = \norm{A_p(s)u(s) - A_p(t)u(s)} \leq C\abs{t-s} \sum_{\abs{\alpha}\leq m} \norm{\partial^\alpha u(s)}\]
for $t,s\in[0,T]$.
Thus, there exists $\widetilde{C}_p\geq0$ such that
\[\norm{(\Id - A_p(t)A_p(s)^{-1})f}\leq \widetilde{C}_p\abs{t-s} \norm{u(s)}_{\mathrm{W}^{p,m}(\RR^d)},\quad t,s\in[0,T].\]
Hence, the uniform bound for $\norm{A_p(s)^{-1}}$ yields existence of $C_p\geq0$ such that
\[\norm{(\Id - A_p(t)A_p(s)^{-1})f}\leq C_p\abs{t-s} \norm{f},\quad f\in \mathrm{L}^p(\RR^d).\]
By Lemma \ref{lemma:EquivLip}(b) we obtain that \eqref{eqn:ResolventLipschitz1} is satisfied. Thus, the first part of Corollary \ref{cor:DS_extra_withoutC} yields that $\bigl(((A_p)_{-1}(t))_{t\in[0,T]},\mathrm{W}^{p',-m}(\RR^d),\mathrm{L}^p(\RR^d)\bigr)$ is a $\CD$-system. Moreover, by \cite[Section 6.1]{Lunardi2013}, there exists a unique evolution family $(U_p(t,s))_{t\geq s}$ associated with $\bigl((A_p(t),\mathrm{W}^{p,m}(\RR^d))\bigr)_{t\in [0,T]}$. Furthermore, since $\mathrm{W}^{p',-m}(\RR^d)$ is reflexive (see \cite[Theorem 3.3.12]{AdamsFournier2003}), we obtain $F_0 = \dom((A_p)_{-1}(t))$ and therefore $F_0 = \mathrm{L}^p(\RR^d)$ for all $t\in [0,T]$ by \cite[Corollary II.5.21]{EN}.
Hence, the second part of Corollary \ref{cor:DS_extra_withoutC} yields that if $$B(\cdot)\in \mathrm{L}^{\infty}\bigl([0,T],\LLL_{\mathrm{s}}(\mathrm{L}^p(\RR^d),\mathrm{L}^p(\RR^d))\bigr)\cap\Lip\bigl([0,T],\LLL_{\mathrm{s}}(\mathrm{L}^p(\RR^d),\mathrm{W}^{p',-m}(\RR^d))\bigr)$$ then there exists an evolution system $(V_p(t,s))_{t\geq s}$ on $X$ giving rise to so-called mild solutions for the perturbed non-autonomous Cauchy problem
\begin{align*}
    u'(t) & = (A(t)+B(t))u(t),\quad t\in [s,T],\\
    u(s) & = u_s\in \mathrm{L}^p(\RR^d).
\end{align*}

\begin{example}
    As a particular example, for $j\in\{1,\ldots,d\}$ let $a_{2e_j}\from[0,T]\to \CC$ be Lipschitz continuous with $c:=-\max_{j\in\{1,\ldots,d\}} \max_{t\in[0,T]} \Re a_{2e_j}(t)>0$, and for $\alpha\in\NN_0^d$ with $\abs{\alpha}\leq 1$ let $a_\alpha\from [0,T]\to\CC$ be Lipschitz continuous. Let
    \[a(t,\xi):= \sum_{j=1}^d a_{2e_j}(t) i^2\xi_j^2 + \sum_{\abs{\alpha}\leq 1} a_\alpha(t) (i\xi)^\alpha, \quad t\in[0,T], \xi\in \RR^d,\]
    and hence $A_p(t) = -\sum_{j=1}^d a_{2e_j}(t)\partial_j^2 - \sum_{\abs{\alpha}\leq 1} a_\alpha(t) \partial^\alpha$ for $t\in[0,T]$.
    As above, let us assume that there exists $\omega>0$ such that $\Re a(t,\xi)\geq\omega$ for all $t\in[0,T]$, $\xi\in\RR^d$.
    Then Assumptions \ref{assump:GenAnalySemi}--\ref{assump:ConstExtra} are satisfied, the resolvents of $(A_p(t))_{t\in[0,T]}$ are commuting and \eqref{eqn:ResolventLipschitz1} are satisfied. For $t\in[0,T]$ and $f\in \mathrm{L}^p(\RR^d)$ let
    \[B(t)f:=\begin{cases} f & t=0,\\
              \frac{1}{(2t)^d} \mathbf{1}_{(-t,t)^d} \ast f & t>0.
             \end{cases}\]
    Then clearly $B(\cdot)\in \mathrm{C}\bigl([0,T],\LLL_{\mathrm{s}}(\mathrm{L}^p(\RR^d),\mathrm{L}^p(\RR^d))\bigr)$ (however, $B(\cdot)$ is not Lipschitz continuous). Moreover, a short calculation reveals that $$B(\cdot)\in \Lip\bigl([0,T],\LLL_{\mathrm{s}}(\mathrm{L}^p(\RR^d),\mathrm{W}^{p',-2}(\RR^d))\bigr).$$
    Thus, Corollary \ref{cor:DS_extra_withoutC} is applicable.
\end{example}

\section*{Acknowledgement}
C.B. thanks the University of Wuppertal for the possibility of funding a stay at the University of Salerno within the Erasmus exchange program as well as A.~Rhandi for fruitful discussions and valuable comments. Moreover, C.B. acknowledges funding by the Deutsche Forschungsgemeinschaft (DFG, German Research Foundation) - 468736785. The authors are also thankful for the fruitful comments of the anonymous referees which allowed us to improve the article.


\begin{thebibliography}{10}

\bibitem{Acquistapace1988}
P.~Acquistapace.
\newblock Evolution operators and strong solutions of abstract linear parabolic
  equations.
\newblock {\em Differential Integral Equations}, 1(4):433--457, 1988.

\bibitem{AT1985}
P.~Acquistapace and B.~Terreni.
\newblock On the abstract nonautonomous parabolic {C}auchy problem in the case
  of constant domains.
\newblock {\em Ann. Mat. Pura Appl. (4)}, 140:1--55, 1985.

\bibitem{Acquistapace1987}
P.~Acquistapace and B.~Terreni.
\newblock A unified approach to abstract linear nonautonomous parabolic
  equations.
\newblock {\em Rend. Sem. Mat. Univ. Padova}, 78:47--107, 1987.

\bibitem{AdamsFournier2003}
R.~A. Adams and J.~J.~F. Fournier.
\newblock {\em Sobolev spaces}, volume 140 of {\em Pure and Applied Mathematics
  (Amsterdam)}.
\newblock Elsevier/Academic Press, Amsterdam, second edition, 2003.

\bibitem{Ama1988}
H.~Amann.
\newblock Parabolic evolution equations in interpolation and extrapolation
  spaces.
\newblock {\em J. Funct. Anal.}, 78(2):233--270, 1988.

\bibitem{Ama1995}
H.~Amann.
\newblock {\em Linear and quasilinear parabolic problems. {V}ol. {I}},
  volume~89 of {\em Monographs in Mathematics}.
\newblock Birkh\"{a}user Boston, Inc., Boston, MA, 1995.
\newblock Abstract linear theory.

\bibitem{B2002}
S.~Bertoni.
\newblock Evolutionary co-processes and perturbations of {K}ato-stable families
  of operators.
\newblock {\em Dyn. Contin. Discrete Impuls. Syst. Ser. A Math. Anal.},
  9(2):225--235, 2002.

\bibitem{B2005}
S.~Bertoni.
\newblock Multiplicative perturbations of constant-domain evolution equations.
\newblock {\em J. Evol. Equ.}, 5(2):291--316, 2005.

\bibitem{B2014}
S.~Bertoni.
\newblock Stability of {CD}-systems under perturbations in the {F}avard class.
\newblock {\em Mediterr. J. Math.}, 11(4):1195--1204, 2014.

\bibitem{BombachGabelSeifertTautenhahn2021}
C.~Bombach, F.~Gabel, C.~Seifert, and M.~Tautenhahn.
\newblock Observability for non-autonomous systems.
\newblock in preparation.

\bibitem{DaPG1984}
G.~Da~Prato and P.~Grisvard.
\newblock Maximal regularity for evolution equations by interpolation and
  extrapolation.
\newblock {\em J. Funct. Anal.}, 58(2):107--124, 1984.

\bibitem{DoreVenni1987}
G.~Dore and A.~Venni.
\newblock On the closedness of the sum of two closed operators.
\newblock {\em Math. Z.}, 196(2):189--201, 1987.

\bibitem{EN}
K.-J. Engel and R.~Nagel.
\newblock {\em One-parameter semigroups for linear evolution equations}, volume
  194 of {\em Graduate Texts in Mathematics}.
\newblock Springer-Verlag, New York, 2000.
\newblock With contributions by S. Brendle, M. Campiti, T. Hahn, G. Metafune,
  G. Nickel, D. Pallara, C. Perazzoli, A. Rhandi, S. Romanelli and R.
  Schnaubelt.

\bibitem{Evans1976}
D.~E. Evans.
\newblock Time dependent perturbations and scattering of strongly continuous
  groups on {B}anach spaces.
\newblock {\em Math. Ann.}, 221(3):275--290, 1976.

\bibitem{GalleratiVeraar2017}
C.~Gallarati and M.~Veraar.
\newblock Maximal regularity for non-autonomous equations with measurable
  dependence on time.
\newblock {\em Potential Anal.}, 46(3):527--567, 2017.

\bibitem{G2017}
J.~A. Goldstein.
\newblock {\em Semigroups of linear operators \& applications}.
\newblock Dover Publications, Inc., Mineola, NY, 2017.
\newblock Second edition of [ MR0790497], Including transcriptions of five
  lectures from the 1989 workshop at Blaubeuren, Germany.

\bibitem{Gr1987}
G.~Greiner.
\newblock Perturbing the boundary conditions of a generator.
\newblock {\em Houston J. Math.}, 13(2):213--229, 1987.

\bibitem{Haase2006}
M.~Haase.
\newblock {\em The functional calculus for sectorial operators}, volume 169 of
  {\em Operator Theory: Advances and Applications}.
\newblock Birkh\"{a}user Verlag, Basel, 2006.

\bibitem{HMR2015}
S.~Hadd, R.~Manzo, and A.~Rhandi.
\newblock Unbounded perturbations of the generator domain.
\newblock {\em Discrete Contin. Dyn. Syst.}, 35(2):703--723, 2015.

\bibitem{HP1957}
E.~Hille and R.~S. Phillips.
\newblock {\em Functional analysis and semi-groups}.
\newblock American Mathematical Society Colloquium Publications, vol. 31.
  American Mathematical Society, Providence, R. I., 1957.
\newblock rev. ed.

\bibitem{HP1994}
D.~Hinrichsen and A.~J. Pritchard.
\newblock Robust stability of linear evolution operators on {B}anach spaces.
\newblock {\em SIAM J. Control Optim.}, 32(6):1503--1541, 1994.

\bibitem{JL2021}
B.~Jacob and H.~Laasri.
\newblock Well-posedness of infinite-dimensional non-autonomous passive
  boundary control systems.
\newblock {\em Evol. Equ. Control Theory}, 10(2):385--409, 2021.

\bibitem{K1953}
T.~Kato.
\newblock Integration of the equation of evolution in a {B}anach space.
\newblock {\em J. Math. Soc. Japan}, 5:208--234, 1953.

\bibitem{Kato1961}
T.~Kato.
\newblock Abstract evolution equations of parabolic type in {B}anach and
  {H}ilbert spaces.
\newblock {\em Nagoya Math. J.}, 19:93--125, 1961.

\bibitem{K1970}
T.~Kato.
\newblock Linear evolution equations of ``hyperbolic'' type.
\newblock {\em J. Fac. Sci. Univ. Tokyo Sect. I}, 17:241--258, 1970.

\bibitem{K1973}
T.~Kato.
\newblock Linear evolution equations of ``hyperbolic'' type. {II}.
\newblock {\em J. Math. Soc. Japan}, 25:648--666, 1973.

\bibitem{K1985}
T.~Kato.
\newblock {\em Abstract differential equations and nonlinear mixed problems}.
\newblock Lezioni Fermiane. [Fermi Lectures]. Scuola Normale Superiore, Pisa;
  Accademia Nazionale dei Lincei, Rome, 1985.

\bibitem{KT1962}
T.~Kato and H.~Tanabe.
\newblock On the abstract evolution equation.
\newblock {\em Osaka Math. J.}, 14:107--133, 1962.

\bibitem{KP2013}
P.~E. Kloeden and C.~P\"{o}tzsche, editors.
\newblock {\em Nonautonomous dynamical systems in the life sciences}, volume
  2102 of {\em Lecture Notes in Mathematics}.
\newblock Springer, Cham, 2013.
\newblock Mathematical Biosciences Subseries.

\bibitem{Lunardi2013}
A.~Lunardi.
\newblock {\em Analytic semigroups and optimal regularity in parabolic
  problems}.
\newblock Modern Birkh\"{a}user Classics. Birkh\"{a}user/Springer Basel AG,
  Basel, 1995.
\newblock [2013 reprint of the 1995 original] [MR1329547].

\bibitem{L2018}
A.~Lunardi.
\newblock {\em Interpolation theory}, volume~16 of {\em Appunti. Scuola Normale
  Superiore di Pisa (Nuova Serie) [Lecture Notes. Scuola Normale Superiore di
  Pisa (New Series)]}.
\newblock Edizioni della Normale, Pisa, 2018.
\newblock Third edition [of MR2523200].

\bibitem{MS2008}
L.~Maniar and R.~Schnaubelt.
\newblock Robustness of {F}redholm properties of parabolic evolution equations
  under boundary perturbations.
\newblock {\em J. Lond. Math. Soc. (2)}, 77(3):558--580, 2008.

\bibitem{MonniauxPruess1997}
S.~Monniaux and J.~Pr\"{u}ss.
\newblock A theorem of the {D}ore-{V}enni type for noncommuting operators.
\newblock {\em Trans. Amer. Math. Soc.}, 349(12):4787--4814, 1997.

\bibitem{NagelExtra}
R.~Nagel.
\newblock Extrapolation spaces for semigroups (nonlinear evolution equations
  and applications).
\newblock {\em RIMS K\^{o}ky\^{u}roku}, 1009:181--191, 8 1997.

\bibitem{NS1994}
R.~Nagel and E.~Sinestrari.
\newblock Inhomogeneous {V}olterra integrodifferential equations for
  {H}ille-{Y}osida operators.
\newblock In {\em Functional analysis ({E}ssen, 1991)}, volume 150 of {\em
  Lecture Notes in Pure and Appl. Math.}, pages 51--70. Dekker, New York, 1994.

\bibitem{NaSi2006}
R.~Nagel and E.~Sinestrari.
\newblock Extrapolation spaces and minimal regularity for evolution equations.
\newblock {\em J. Evol. Equ.}, 6(2):287--303, 2006.

\bibitem{NickelPhD}
G.~Nickel.
\newblock {\em On evolution semigroups and wellposedness of nonautonomous
  {C}auchy problems}.
\newblock PhD thesis, Eberhard-{K}arls-{U}niversit\"{a}t T\"{u}bingen, 1996.

\bibitem{P1983}
A.~Pazy.
\newblock {\em Semigroups of linear operators and applications to partial
  differential equations}, volume~44 of {\em Applied Mathematical Sciences}.
\newblock Springer-Verlag, New York, 1983.

\bibitem{RRS1996}
F.~R\"{a}biger, A.~Rhandi, and R.~Schnaubelt.
\newblock Perturbation and an abstract characterization of evolution
  semigroups.
\newblock {\em J. Math. Anal. Appl.}, 198(2):516--533, 1996.

\bibitem{RSRV2000}
F.~R\"{a}biger, R.~Schnaubelt, A.~Rhandi, and J.~Voigt.
\newblock Non-autonomous {M}iyadera perturbations.
\newblock {\em Differential Integral Equations}, 13(1-3):341--368, 2000.

\bibitem{Rhandi1997}
A.~Rhandi.
\newblock Extrapolation methods to solve non-autonomous retarded partial
  differential equations.
\newblock {\em Studia Math.}, 126(3):219--233, 1997.

\bibitem{Sch1996}
R.~Schnaubelt.
\newblock {\em Exponential bounds and hyperbolicity of evolution families}.
\newblock PhD thesis, Eberhard--Karls--Universit\"{a}t T\"{u}bingen, 1996.

\bibitem{Sch1999}
R.~Schnaubelt.
\newblock Sufficient conditions for exponential stability and dichotomy of
  evolution equations.
\newblock {\em Forum Math.}, 11(5):543--566, 1999.

\bibitem{S2002}
R.~Schnaubelt.
\newblock Feedbacks for nonautonomous regular linear systems.
\newblock {\em SIAM J. Control Optim.}, 41(4):1141--1165, 2002.

\bibitem{Tanabe1960}
H.~Tanabe.
\newblock On the equations of evolution in a {B}anach space.
\newblock {\em Osaka Math. J.}, 12:363--376, 1960.

\bibitem{Tanabe1961}
H.~Tanabe.
\newblock Evolutional equations of parabolic type.
\newblock {\em Proc. Japan Acad.}, 37:610--613, 1961.

\bibitem{T1979}
H.~Tanabe.
\newblock {\em Equations of evolution}, volume~6 of {\em Monographs and Studies
  in Mathematics}.
\newblock Pitman (Advanced Publishing Program), Boston, Mass.-London, 1979.
\newblock Translated from the Japanese by N. Mugibayashi and H. Haneda.

\bibitem{Xiao2002}
T.-J. Xiao, J.~Liang, and J.~van Casteren.
\newblock Time dependent {D}esch-{S}chappacher type perturbations of {V}olterra
  integral equations.
\newblock {\em Integral Equations Operator Theory}, 44(4):494--506, 2002.

\end{thebibliography}
\end{document}